\documentclass{amsart}
\usepackage{amsmath,amssymb,tikz}
\usetikzlibrary{arrows,shapes,positioning,decorations.markings}
\usepackage[margin=1.5cm]{geometry}
\newtheorem{theorem}{Theorem}[section]
\newtheorem{lemma}[theorem]{Lemma}

\newtheorem{conjecture}[theorem]{Conjecture}
\newtheorem{proposition}[theorem]{Proposition}

\newtheorem{definition}[theorem]{Definition}

\newcommand{\Sym}{\mathop{\mathrm{Sym}}}
\newcommand{\Cay}{\mathop{\Gamma}}
\newcommand{\Aut}{\mathop{\mathrm{Aut}}}

\numberwithin{equation}{section}
\def\Z#1{{\bf Z}(#1)}
\def\cent#1#2{{\bf C}_{#1}(#2)}
\def\norm#1#2{{\bf N}_{#1}(#2)}
\def\Dic{{\rm Dic}}
\makeatletter
\newcommand{\myitem}[1]{%
\item[#1]\protected@edef\@currentlabel{#1}%
}
\makeatother

\begin{document}

\title[Babai-Godsil vs Xu]{On the equivalence between a conjecture of Babai-Godsil and  a conjecture of Xu concerning the  enumeration of Cayley graphs}

\author[P. Spiga]{Pablo Spiga}
\address{Pablo Spiga,
Dipartimento di Matematica e Applicazioni, University of Milano-Bicocca,\newline
Via Cozzi 55, 20125 Milano, Italy}\email{pablo.spiga@unimib.it}

\thanks{Address correspondence to P. Spiga, E-mail: pablo.spiga@unimib.it.}

\begin{abstract}
In this paper we show that two distinct conjectures, the first proposed by Babai and Godsil in $1982$ and the second proposed by Xu in $1998$, concerning the asymptotic enumeration of Cayley graphs are in fact equivalent. This result follows from a more general theorem concerning the asymptotic enumeration of a certain family of Cayley graphs.
\end{abstract}

\maketitle

\section{Introduction}\label{sec:introduction}
All digraphs and groups considered in this paper are finite. A \textbf{\textit{digraph}} $\Gamma$ is an ordered pair $(V, A)$ where the \textbf{\textit{vertex-set}} $V$ is a finite non-empty set and the \textbf{\textit{arc-set}} $A \subseteq V \times V$ is a binary relation
on $V$. The elements of $V$ and $A$ are called \textbf{\textit{vertices}} and \textbf{\textit{arcs}} of $\Gamma$, respectively. An \textbf{\textit{automorphism}} of $\Gamma$ is a permutation $\sigma$
of $V$ with $A^\sigma=A$, that is, $(x^\sigma , y^\sigma ) \in A$ for every $(x, y) \in A$.
Let $R$ be a group and let $S$ be a subset of $R$. The \textbf{\textit{Cayley digraph}} on $R$ with connection set $S$ (which we   denote by $\Cay(R, S)$) is the
digraph with vertex-set $R$ and with $(g, h)$ being an arc if and only if $hg^{-1} \in S$. The group $R$ acts regularly as a group of
automorphisms of $\Cay(R, S)$ by right multiplication and hence $R \le \Aut(\Cay(R, S))$.
When  $R=\Aut(\Cay(R, S)$, the digraph $\Gamma$ is called a \textbf{\textit{DRR}} (for digraphical regular representation).
Babai and Godsil made the following conjecture.

\begin{conjecture}[\cite{Go2}, Conjecture 3.13; \cite{BaGo}]\label{digraphmainconjecturee}
Let $R$ be a group of order $r$. The proportion of subsets $S$ of $R$ such that $\Cay(R,S)$ is a $\mathrm{DRR}$ goes to $1$ as $r\to\infty$.
\end{conjecture}
This conjecture has been recently proved in~\cite{MSMS}. This paper is the first step for proving yet another conjecture of Babai and Godsil concerning the enumeration of \textbf{\textit{Cayley graphs}}.  Recall that $\Cay(R,S)$ is undirected if and only if $S$ is \textbf{\textit{inverse-closed}}, that is, $S^{-1}:=\{s^{-1}\mid s\in S\}=S$. While the number of Cayley digraphs on $R$ is $2^{|R|}$, which is a number that depends on the cardinality of $R$ only, the number of undirected Cayley graphs on $R$ is $2^{\frac{|R|+|{\bf I}(R)|}{2}}$ (see Lemma~\ref{lemma111}), where ${\bf I}(R):=\{\iota\in R\mid \iota^2=1\}$, and hence depends on the algebraic structure of $R$.
 
Although the difference between Cayley digraphs and Cayley graphs seems only minor and to some extent only aesthetic, the behaviour between these two classes of combinatorial objects with respect to their automorphisms can be dramatically different. For instance, it was proved by Babai~\cite[Theorem~$2.1$]{babai11} that, except for $$Q_8,\,C_2\times C_2,\,C_2\times C_2\times C_2,\, C_2\times C_2\times C_2\times C_2\, \textrm{ and }C_3\times C_3,$$   every finite group $R$ admits a DRR. Borrowing a phrase which I once heard from Tom Tucker: ``besides some low level noise, every finite group admits a DRR''. The analogue for GRRs is not the same. Indeed, it turns out that there are two (and only two) infinite families of groups that do no admit GRRs. The first family consists of abelian groups of exponent greater than two. If $R$ is such a group and $\iota$ is the automorphism of $R$ mapping every element to its inverse, then every Cayley graph on $R$ admits $R\rtimes\langle \iota\rangle$ as a group of automorphisms. Since $R$ has exponent greater than $2$, $\iota\ne 1$ and hence no Cayley graph on $R$ is a GRR. The other family of groups that do not admit GRRs are the generalised dicyclic groups, see \cite[Definition $1.1$]{MSV} for a definition and also Definition~\ref{defeq:2} below. These two families were discovered by Mark Watkins~\cite{Watkins22}.
 
 It was proved by Godsil~\cite{Godsilnon} that abelian groups of exponent greater than $2$ and generalised dicyclic groups are the only two infinite families of groups that do not admit GRRs. (A lot of papers have been published for determining those groups admitting a GRR, and some
of the most influential works along the way appeared in~\cite{1010,1313, 1414, 1717, 1818,1919}.) The stronger Conjecture~\ref{conjecture....} was made (at various times) by Babai, Godsil, Imrich and Lov\'{a}sz.
 
 \begin{conjecture}[see \cite{BaGo}, Conjecture $2.1$ and \cite{Go2}, Conjecture $3.13$]\label{conjecture....}Let $R$ be a group of order $r$ which is neither generalized dicyclic nor abelian of exponent greater than $2$. The proportion of inverse-closed subsets $S$ of $R$ such that $\Cay(R,S)$ is a $\mathrm{GRR}$ goes to $1$ as $r\to \infty$.
 \end{conjecture}
This conjecture is open at the moment and some of the techniques developed in~\cite{MSMS} for dealing with digraphs are not suited for dealing with undirected graphs.  

The scope of this paper is twofold. Broadly speaking, we aim to start a long process where we try to generalize and adapt the results obtained in~\cite{MSMS} for eventually dealing with undirected graphs and proving Conjecture~\ref{conjecture....}. We start this process by dealing with the \textit{\textbf{first natural obstruction}} for the existence of GRRs. 
Given an inverse-closed subset $S$ of $R$, the set $S$ fails to give rise to a GRR
essentially for two different reasons. (Let $A := \Aut(\Cay(R, S))$.)
\begin{enumerate}
\item There are non-identity group automorphisms of $R$ leaving the set $S$ invariant. This case arises when $\norm  A R>$ (this is the typical obstruction and we have encountered this obstruction already when we briefly discussed abelian groups of exponent greater than $2$).
\item The only group automorphism of $R$ leaving the set $S$ invariant is the identity  and there are some automorphisms of $\Cay(R, S)$ not lying in $R$. This case arises when $\norm  A R = R$ and $A > R$ : this obstruction is somehow
mysterious and much harder to analyze.
\end{enumerate}
These two obstructions are clear (if not obvious) to readers familiar with the enumeration problem of Cayley graphs~\cite{MSMS} and in
particular to readers familiar with~\cite{BaGo}. Actually the same obstructions arise in the enumeration problem of other types of Cayley graphs, for instance in the asymptotic enumeration of DFRs~\cite{DFR} and GFRs~\cite{DTW,GFR} and in the recent solution of the GFR conjecture~\cite{GFR}. In this paper we deal with the first obstruction. 
\begin{theorem}\label{thrm:1}Let $R$ be a group of order $r$ which is neither generalized dicyclic nor abelian of exponent greater than $2$. The proportion of inverse-closed subsets $S$ of $R$ such that $\norm{\Aut(\Cay(R,S))}R>R$ goes to $0$ as $r\to \infty$.
\end{theorem}
We observe that in Proposition~\ref{propo:aut} we have a quantified version of Theorem~\ref{thrm:1}. Moreover, in Lemma~\ref{l:aut} we have a more technical version of Theorem~\ref{thrm:1} which includes also generalized dicyclic groups and abelian groups of exponent greater than $2$. These two more techinical results are in our opinion needed to follow the footsteps of the argument in~\cite{MSMS} for the asymptotic enumeration of Cayley digraphs.

The second scope of this paper is to prove that a famous conjecture of Xu on the asymptotic enumeration of normal Cayley graphs is actually equivalent to Conjecture~\ref{conjecture....}. A Cayley (di)graph $\Gamma$ on $R$ is said to be a \textit{\textbf{normal Cayley (di)graph on}} $R$ if the regular
representation of $R$ is normal in $\Aut(\Gamma)$, that is, $R\unlhd \Aut(\Gamma)$. Clearly, every DRR and every GRR $\Gamma$ on $R$ is a normal Cayley (di)graph because $R=\Aut(\Gamma)$. Xu has conjectured that almost all Cayley (di)graphs on $R$ are normal Cayley
(di)graphs on $R$.
\begin{conjecture}[see Conjecture 1~\cite{Xu1998}]\label{conjectureXu} The minimum, over all groups $R$ of order $r$, of the proportion of inverse-closed subsets $S$ of $R$ such
that $\Cay(R, S)$ is a normal Cayley digraph tends to $1$ as $r \to \infty$.
\end{conjecture}
Xu has formulated an analogous conjecture for Cayley digraphs and this
version was shown to be true in~\cite{MSMS} by proving the stronger Conjecture~\ref{digraphmainconjecturee}. The veracity of Conjecture~\ref{conjectureXu} when $R$ is an abelian group
and when $R$ is a dicyclic group was proved in~\cite{DSV, MSV}. In this paper we show that Conjecture~\ref{conjecture....} and Conjecture~\ref{conjectureXu} are actually equivalent.
\begin{theorem}\label{equivalence}Conjecture~$\ref{conjecture....}$ holds true if and only if Conjecture~$\ref{conjectureXu}$ holds true.
\end{theorem}

\section{Group automorphisms}\label{automorphisms}
\begin{definition}\label{defeq:1}
{\rm 
Given a finite group $R$ and $x\in R$, we let $o(x)$ denote the order of the element $x$ and we let ${\bf I}(R):=\{x\in R\mid o(x)\le 2\}$ be the set of elements of $R$ having order at most $2$. We let $\mathbf{c}(R)$ denote the fraction $(|R|+|{\bf I}(R)|)/2$, that is,
$$\mathbf{c}(R)=\frac{|R|+|{\bf I}(R)|}{2}.$$
Given a subset $X$ of $R$, we write ${\bf I}(X):=X\cap {\bf I}(R)$. Finally, we denote by $\Z R$ the centre of $R$.}
\end{definition}

\begin{lemma}\label{lemma111}Let $R$ be a finite group. The number of inverse-closed subsets $S$ of $R$ is $2^{\mathbf{c}(R)}.$ 
\end{lemma}
\begin{proof}
 Given an arbitrary inverse-closed subset $S$ of $R$, $S\cap {\bf I}(R)$ is an arbitrary subset of ${\bf I}(R)$ whereas in $S\cap (R\setminus {\bf I}(R))$ the elements come in pairs, where each element is paired up to its inverse. Thus the number of inverse-closed subsets of $R$ is $$2^{|{\bf I}(R)|}\cdot 2^{\frac{|R\setminus {\bf I}(R)|}{2}}=2^{{\bf c}(R)}.\qedhere$$
\end{proof}
 
\begin{definition}\label{defeq:2}{\rm Let $R$ be a finite group. Given an automorphism $\varphi$ of  $R$, we set
\begin{align*}
\cent R \varphi&:=\{x\in R\mid x^\varphi=x\},\\
\cent R \varphi^{-1}&:=\{x\in R\mid x^\varphi=x^{-1}\}.\\
\end{align*}
Observe that, when $\varphi=id_R$ is the identity automorphism of $R$, $\cent R{\varphi}^{-1}={\bf I}(R)$.

Given  $x\in R$, we denote by $\iota_x:R\to R$ the inner automorphism of $R$ induced by $x$, that is, $m^{\iota_x}=xmx^{-1}$, for every $m\in R$. (Usually, the automorphism $\iota_x$ is defined by $m\mapsto m^{\iota_x}=x^{-1}mx$, however for our application it is more convenient to define $\iota_x$ by $m\mapsto m^{\iota_x}=xmx^{-1}$.) When $A\unlhd R$, we still denote by $\iota_x$ the restriction to $A$ of the automorphism $\iota_x$, this makes the notation not too cumbersome to use and hopefully will cause no confusion.

Finally, we let $\iota:R\to R$ be the permutation defined by $x^\iota=x^{-1}$, for every $x\in R$. In particular, when $R$ is abelian, $\iota$ is an automorphism of $R$. Furthermore, $\iota=id_R$ if and only if $R$ is an abelian group of exponent at most $2$.
}
\end{definition}

\begin{definition}\label{defeq:2}{\rm
Let $A$ be an abelian group of even order and of exponent greater than $2$, and let $y$ be an involution of $A$. The \emph{generalised dicyclic group} $\Dic(A,y,x)$ is the group $\langle A,x\mid x^2=y, a^x=a^{-1}, \forall a \in A\rangle$. A group is called \emph{generalised dicyclic} if it is isomorphic to some $\Dic(A,y,x)$. When $A$ is cyclic, $\Dic(A,y,x)$ is called a \emph{dicyclic} or \emph{generalised quaternion group}.

We let $\bar{\iota}_A:\Dic(A,y,x)\to \Dic(A,y,x)$ be the mapping defined by $(ax)^{\bar{\iota}_A}=ax^{-1}$ and $a^{\bar{\iota}_A}=a$, for every $a\in A$. In particular, $\bar{\iota}_A$ is an automorphism of $\Dic(A,y,x)$.
The role of the label ``$A$'' in $\bar{\iota}_A$ seems unnecessary, however we use this label to stress one important fact. An abstract group $R$ might be isomorphic to $\Dic(A,y,x)$, for various choices of $A$. Therefore, since the automorphism $\bar{\iota}_A$ depends on $A$ and since we might have more than one choice of $A$, we prefer a notation that emphasizes this fact.}
\end{definition}
\begin{lemma}\label{l:aut0}
Let $R$ be a finite group and let $\varphi$ be an automorphism of $R$ with $|R:\cent R\varphi|=2$. Then one of the following holds:
\begin{enumerate}
\item\label{eq:aut0} $\frac{1}{4}(|R|+|{\bf I}(R)|+|\cent R\varphi|+|\cent R\varphi^{-1}|)\le \mathbf{c}(R)-\frac{|R|}{32}$, 
\item\label{eq:aut1} $R$ is generalized dicyclic over the abelian group $\cent R\varphi$ and $\varphi=\bar{\iota}_{\cent R\varphi}$,
\item\label{eq:aut2} $R$ is abelian of exponent greater than $2$ and $\varphi=\iota$.
\end{enumerate}
\end{lemma}
\begin{proof}
For simplicity, we let $A:=\cent R\varphi$ and we let $o$ denote the left-hand side in~\eqref{eq:aut0}.

 Suppose that $\cent R\varphi^{-1}\subseteq A$. Then $\cent R\varphi^{-1}=\cent A\varphi^{-1}=\cent A{id_A}^{-1}={\bf I}(A)$. Thus
\begin{eqnarray*}
o=\frac{1}{4}\left(|R|+|{\bf I}(R)|+\frac{|R|}{2}+|{\bf I}(A)|\right)\le \frac{1}{4}\left(\frac{3}{2}|R|+2|{\bf I}(R)|\right)=\frac{|R|+|{\bf I}(R)|}{2}-\frac{|R|}{8}=\mathbf{c}(R)-\frac{|R|}{8},
\end{eqnarray*}
and~\eqref{eq:aut0} holds in this case.

Suppose that $\cent R\varphi^{-1}\nsubseteq A$. In particular, there exists $x\in R\setminus A$ with $x^\varphi=x^{-1}$. As $|R:A|=2$, we have $R=A\cup Ax$. For every $a\in A$, since $\varphi$ is an automorphism of $R$ fixing point-wise $A$ and since $xax^{-1}\in A$, we deduce
$$xax^{-1}=(xax^{-1})^\varphi=x^\varphi a^\varphi (x^{-1})^\varphi=x^{-1}ax$$
and hence $x^2a=ax^2$, that is, $x^2\in \Z {\langle A,x\rangle}=\Z R$. As $x^2\in A$, we have $x^2=(x^2)^\varphi=(x^\varphi)^2=(x^{-1})^2=x^{-2}$, that is, $x^4=1$. Summing up,
\begin{equation}\label{eq:art}
x^2\in\Z R,\quad x^4=1.
\end{equation}

Now, let $y\in\cent R\varphi^{-1}\setminus A$. Then, $y=ax$, for some $a\in A$. Moreover, $y^\varphi=y^{-1}=(ax)^{-1}=x^{-1}a^{-1}$ and $y^\varphi=(ax)^\varphi=a^\varphi x^{\varphi}=ax^{-1}$. Thus 
$$x^{-1}a^{-1}=ax^{-1},$$
that is, $xax^{-1}=a^{-1}$. Recall that $\iota_x:A\to A$ is the restriction to the normal subgroup $A$ of the  inner automorphism of $R$ determined by $x$, that is, $a^{\iota_x}=xax^{-1}$, for every $a\in A$. We have shown that  $\cent R\varphi^{-1}\setminus A=\cent A{\iota_x}^{-1}x$. As $\cent R\varphi^{-1}\cap A={\bf I}(A)$, we get 
\begin{equation}\label{eq:art1}
\cent R\varphi^{-1}={\bf I}(A)\cup \cent A{\iota_x}^{-1}x
\end{equation} and $|\cent R\varphi^{-1}|=|{\bf I}(A)|+|\cent A{\iota_x}^{-1}|$.

Suppose that $|\cent A{\iota_x}^{-1}|\le 3|A|/4$.  Thus, by~\eqref{eq:art1}, we have
\begin{eqnarray*}
o&=&\frac{1}{4}\left(\frac{3}{2}|R|+|{\bf I}(R)|+|{\bf I}(A)|+|\cent A{\iota_x}^{-1}|\right)\le \frac{1}{4}\left(\frac{3}{2}|R|+2|{\bf I}(R)|+\frac{3|A|}{4}\right)\\
&=&\frac{1}{4}\left(\frac{3}{2}|R|+2|{\bf I}(R)|+\frac{3|R|}{8}\right)=
\frac{|R|+|{\bf I}(R)|}{2}-\frac{|R|}{32}=\mathbf{c}(R)-\frac{|R|}{32},
\end{eqnarray*}
and~\eqref{eq:aut0} holds in this case.

Suppose that $|\cent A{\iota_x}^{-1}|>3|A|/4$, that is, the automorphism $\iota_x$ of $A$ inverts more than $3/4$ of its elements. By a result of Miller~\cite{Miller}, $A$ is abelian. Since $A$ is abelian, it is easy to verify that $\cent A{\iota_x}^{-1}$ is a subgroup of $A.$ As $|\cent A{\iota_x}^{-1}|>3|A|/4$, we get $\cent A{\iota_x}^{-1}=A$ and $\iota_x$ acts on $A$ inverting each of its elements. From~\eqref{eq:art1}, we have 
\begin{equation}\label{eq:vania}\cent R\varphi^{-1}={\bf I}(A)\cup Ax.
\end{equation}
 If ${\bf I}(R)\subseteq {\bf I}(A)$, then no element in $Ax$ is an involution and hence $x$ has order $4$ from~\eqref{eq:art}. When $A$ has exponent greater than $2$, we deduce $R\cong \Dic(A,x^2,x)$ is a generalized dicyclic group over $A$, $\varphi=\bar{\iota}_A$ and~\eqref{eq:aut1} holds in this case. When $A$ has exponent at most $2$, we have ${\bf I}(A)=A$ and $\varphi=\iota$. Hence ${\bf I}(R)=A$, $R$ is an abelian group of exponent greater than $2$ and~\eqref{eq:aut2} holds in this case.
Therefore, we may suppose  ${\bf I}(R)\nsubseteq {\bf I}(A)$.

Let $x'\in {\bf I}(R)\setminus A$. Then, $x'=ax$, for some $a\in A$. Then $1=x'^2=(ax)^2=axax=a(xax^{-1})x^2=aa^{-1}x^2=x^2$ and hence $x^2=1$. Now, for every $b\in A$, we have $(bx)^2=bxbx=b(xbx^{-1})=bb^{-1}=1$. This shows ${\bf I}(R)\setminus A=Ax$. Therefore, ${\bf I}(R)={\bf I}(A)\cup Ax$ and hence ${\bf I}(R)=\cent R\varphi^{-1}$ from~\eqref{eq:vania}. We deduce
\begin{eqnarray*}
o&=&\frac{1}{4}\left(\frac{3|R|}{4}+|{\bf I}(R)|+|\cent R\varphi^{-1}|\right)=\frac{1}{4}\left(\frac{3}{2}|R|+2|{\bf I}(R)|\right)=\frac{|R|+|{\bf I}(R)|}{2}-\frac{|R|}{8}=\mathbf{c}(R)-\frac{|R|}{8},
\end{eqnarray*}
and~\eqref{eq:aut0} holds in this case.
\end{proof}

\begin{lemma}\label{l:aut-1}
Let $R$ be a finite group and let $\varphi$ be an automorphism of $R$ with $|R:\cent R\varphi|=3$. Then one of the following holds:
\begin{enumerate}
\item\label{eq:aut-10} $\frac{1}{4}(|R|+|{\bf I}(R)|+|\cent R\varphi|+|\cent R\varphi^{-1}|)\le \mathbf{c}(R)-\frac{|R|}{96}$, 
\item\label{eq:aut-11} $R$ is abelian of exponent greater than $2$ and $\varphi=\iota$.
\end{enumerate}
\end{lemma}
\begin{proof}
For simplicity, we let $A:=\cent R\varphi$ and we let $o$ denote the left-hand side in~\eqref{eq:aut0}. As $|R:A|=3$, we may write $R=A\cup Ax\cup Ax'$, for some $x,x'\in R$.

Suppose that $\cent R\varphi^{-1}\subseteq A\cup Ay$, for some $y\in \{x,x'\}$. Then $$\cent R\varphi^{-1}=(\cent R\varphi^{-1}\cap A)\cup (\cent {R}\varphi^{-1}\cap Ay)={\bf I}(A)\cup (\cent {R}\varphi^{-1}\cap Ay)\subseteq {\bf I}(A)\cup Ay,$$ because $\varphi$ fixes each element of $A$. Thus
 $|\cent R\varphi^{-1}|\le |{\bf I}(R)|+|A|$ and 
\begin{eqnarray*}
o\le\frac{1}{4}\left(\frac{4}{3}|R|+|{\bf I}(R)|+|{\bf I}(R)|+|A|\right)\le \frac{1}{4}\left(\frac{4}{3}|R|+2|{\bf I}(R)|+\frac{|R|}{3}\right)=\frac{|R|+|{\bf I}(R)|}{2}-\frac{|R|}{12}=\mathbf{c}(R)-\frac{|R|}{12},
\end{eqnarray*}
and~\eqref{eq:aut-10} holds in this case.

Therefore we may suppose that $\cent R\varphi^{-1}\cap Ax\ne\emptyset$ and $\cent R\varphi^{-1}\cap Ax'\ne \emptyset$. In particular, replacing $x$ and $x'$ if necessary, we may suppose that $x,x'\in\cent R\varphi^{-1}$, that is, $x^\varphi=x^{-1}$ and $x'^\varphi=x'^{-1}$.

\smallskip

\noindent\textsc{Case: $A\unlhd R$.}

\smallskip

\noindent As $R/A$ is cyclic of order $3$, we may assume that $x'=x^{-1}$ and that $x$ has odd order. For every $a\in A$, we have 
$xax^{-1}\in A$ and hence
$$xax^{-1}=(xax^{-1})^\varphi=x^\varphi a^\varphi(x^{-1})^\varphi=x^{-1}ax,$$
that is, $x^2a=ax^2$. Therefore $x^2\in \Z{\langle x,A\rangle}=\Z R$. As $x$ has odd order, we deduce $x\in \Z R$. From this it is easy to deduce that 
\begin{equation}\label{coffee}
\cent R\varphi^{-1}={\bf I}(A)\cup {\bf I}(A)x\cup {\bf I}(A)x^{-1}.
\end{equation} 

Assume that $|{\bf I}(A)|\le 3|A|/4$. Thus, by~\eqref{coffee}, we have
\begin{eqnarray*}
o&=&\frac{1}{4}\left(\frac{4}{3}|R|+|{\bf I}(R)|+|{\bf I}(A)|+|{\bf I}(A)|+|{\bf I}(A)|\right)\le \frac{1}{4}\left(\frac{4}{3}|R|+2|{\bf I}(R)|+2|{\bf I}(A)|\right)\\
&\le&\frac{1}{4}\left(\frac{4}{3}|R|+2|{\bf I}(R)|+2\frac{3|A|}{4}\right) =
\frac{1}{4}\left(\frac{4}{3}|R|+2|{\bf I}(R)|+\frac{|R|}{2}\right)=\frac{1}{4}\left(\frac{11}{6}|R|+2|{\bf I}(R)|\right)\\
&=&
\frac{|R|+|{\bf I}(R)|}{2}-\frac{|R|}{24}=\mathbf{c}(R)-\frac{|R|}{24},
\end{eqnarray*}
and~\eqref{eq:aut-10} holds in this case. 

Assume that $|{\bf I}(A)|>3|A|/4$. By~\cite{Miller}, $A$ is abelian. Thus ${\bf I}(A)$ is a subgroup of $A$ with $|{\bf I}(A)|>3|A|/4$. It follows that  $A$ is an elementary abelian $2$-group. As $x\in \Z R$, we deduce that $R$ is abelian and $\varphi=\iota$; thus~\eqref{eq:aut-11} holds in this case.

\smallskip

\noindent\textsc{Case: $A$ is not normal in $R$.}

\smallskip

\noindent Let $K$ be the core of $A$ in $R$. Then $|R:K|=6$ and $R/K$ is isomorphic to the dihedral group of order $6$. Suppose that $\cent R\varphi^{-1}\cap Ky= \emptyset$, for some $y\in R\setminus A$. As $R\setminus A$ is the union of four $K$-cosets and as $\cent R\varphi^{-1}\cap Ky= \emptyset$, we deduce $|\cent R\varphi^{-1}\cap (R\setminus A)|\le 3|K|$. As $\cent R\varphi^{-1}\cap A={\bf I}(A)$, we get $|\cent R\varphi^{-1}|=|\cent R\varphi^{-1}\cap A|+|\cent R\varphi^{-1}\cap (R\setminus A)|\le |{\bf I}(A)|+3|K|$ and hence
\begin{eqnarray*}
o&\le&\frac{1}{4}\left(\frac{4}{3}|R|+|{\bf I}(R)|+|{\bf I}(A)|+3|K|\right)\le \frac{1}{4}\left(\frac{4}{3}|R|+2|{\bf I}(R)|+3\frac{|R|}{6}\right)\\
&=&\frac{1}{4}\left(\frac{11}{6}|R|+2|{\bf I}(R)|\right)=
\frac{|R|+|{\bf I}(R)|}{2}-\frac{|R|}{24}=\mathbf{c}(R)-\frac{|R|}{24},
\end{eqnarray*}
and~\eqref{eq:aut-10} holds in this case. Thus we may suppose $\cent R\varphi^{-1}\cap Ky\ne \emptyset$, for every $y\in R\setminus A$.

Let $x_1,x_2,x_3,x_4\in R\setminus A$ with $R=A\cup Kx_1\cup Kx_2\cup Kx_3\cup Kx_4$ and with $x_1,x_2,x_3,x_4\in \cent R\varphi^{-1}$. As usual we denote by $\iota_{x_i}:K\to K$ the automorphism of $K$ defined by $k^{\iota_{x_i}}=x_ikx_i^{-1}$, for every $k\in K$. For each $i\in\{1,\ldots,4\}$, let $y\in \cent R\varphi^{-1}\cap Kx_i$. Then $y=kx_i$, for some $k\in K$ and hence
$x_i^{-1}k^{-1}=(kx_i)^{-1}=y^{-1}=y^\varphi=(kx_i)^\varphi=k^\varphi x_i^\varphi=kx_i^{-1}$, that is, $x_ikx_i^{-1}=k^{-1}$ and $k\in \cent K{\iota_{x_i}}^{-1}$. This shows 
\begin{equation}\label{eq:art2}\cent R\varphi^{-1}={\bf I}(A)\cup \cent K{\iota_{x_1}}^{-1}x_1\cup\cent K{\iota_{x_2}}^{-1}x_2\cup\cent K{\iota_{x_3}}^{-1}x_3\cup\cent K{\iota_{x_4}}^{-1}x_4.
\end{equation}

Suppose that $|\cent {K}{\iota_{x_i}}^\varphi|\le 3|K|/4$, for some $i\in \{1,2,3,4\}$. Then
$$|\cent R\varphi^{-1}|\le |{\bf I}(A)|+3|K|+\frac{3|K|}{4}=|{\bf I}(A)|+\frac{15|K|}{4}=|{\bf I}(A)|+\frac{5|R|}{8}.$$
Thus
\begin{eqnarray*}
o&\le &\frac{1}{4}\left(\frac{4}{3}|R|+|{\bf I}(R)|+|{\bf I}(A)|+\frac{5|R|}{8}\right)\le \frac{1}{4}\left(\frac{47}{24}|R|+2|{\bf I}(R)|\right)\\
&=&
\frac{|R|+|{\bf I}(R)|}{2}-\frac{|R|}{96}=\mathbf{c}(R)-\frac{|R|}{96},
\end{eqnarray*}
and~\eqref{eq:aut-10} holds in this case. Therefore, we may suppose that $|\cent K {\iota_{x_i}}^{-1}|>3|K|/4$, for each 
$i\in \{1,2,3,4\}$. The work of Miller~\cite{Miller} shows that $K$ is abelian and that, for every $i\in \{1,2,3,4\}$, $x_i$ acts by conjugation on $K$ by inverting each of its elements.  In particular,~\eqref{eq:art2} becomes
\begin{equation}\label{eq:art3}
\cent R\varphi^{-1}={\bf I}(A)\cup (R\setminus A).
\end{equation}

As $R/K$ is isomorphic to the dihedral group of order $6$, we deduce that there exist $i,j,k\in \{1,2,3,4\}$ with $x_ix_j\in Kx_k$. From the previous paragraph, $x_i,x_j$ and $x_k$ act by conjugation on $K$ by inverting each of its elements. Therefore, for every $y\in K$, we have
$$y^{-1}=y^{x_k}=y^{x_ix_j}=(y^{x_i})^{x_j}=(y^{-1})^{x_j}=y,$$
that is, $y^2=1$. This yields that $K$ is an elementary abelian $2$-group and hence $K\subseteq {\bf I}(A)$. Eq~\eqref{eq:art3} gives $\cent R\varphi^{-1}\supseteq K\cup (R\setminus A)$ and hence $|\cent R\varphi^{-1}|\ge |K|+|R\setminus A|=5|R|/6>3|R|/4$. Again, from the work of Miller~\cite{Miller}, we deduce that $R$ is abelian and $\varphi=\iota$, and~\eqref{eq:aut-11} holds in this case.
\end{proof}

\begin{lemma}\label{l:aut}Let $R$ be a finite group and let $\varphi$ be a non-identity automorphism of $R$. Then, one of the following holds
\begin{enumerate}
\item\label{eq:aut00} the number of $\varphi$-invariant inverse-closed subsets of $R$ is at most $2^{\mathbf{c}(R)-\frac{|R|}{96}}$,
\item\label{eq:aut01}$\cent R\varphi$ is abelian of exponent greater than $2$ and has index $2$ in $R$, $R$ is a generalized dicyclic group over $\cent R\varphi$ and $\varphi=\bar{\iota}_{\cent R\varphi}$,
\item\label{eq:aut02}$R$ is abelian of exponent greater than $2$ and $\varphi=\iota$.
\end{enumerate} 
\end{lemma}
\begin{proof}
Let $p$ be a prime dividing the order of $\varphi$. Clearly, if $S\subseteq R$ is $\varphi$-invariant, then $S$ is also $\varphi^{o(\varphi)/p}$-invariant. Therefore, replacing $\varphi$ by $\varphi^{o(\varphi)/p}$ if necessary, we may suppose that $\varphi$ is an automorphism of $R$ of prime order $p$. 

Recall that $\iota:R\to R$ is the permutation of $R$ defined by $x^\iota=x^{-1}$, for every $x\in R$. Let $H:=\langle \iota,\varphi\rangle\le \Sym(R)$. Clearly, the number of $\varphi$-invariant inverse-closed subsets of $R$ is $2^{o}$, where $o$ is the number of $H$-orbits, that is, $o$ is the number of orbits of $H$ in its action on $R$. From the orbit-counting lemma, we have
\begin{equation}\label{eq:ocl}
o=\frac{1}{|H|}\sum_{h\in H}|\mathrm{Fix}_R(h)|,
\end{equation}
where $\mathrm{Fix}_R(h):=\{x\in R\mid x^h=x\}$ is the fixed-point set of $h$ in its action on $R$. 

For every $x\in R$, we have $x^{\iota \varphi}=(x^{-1})^\varphi=(x^\varphi)^{-1}=x^{\varphi\iota}$ and hence $\iota\varphi=\varphi\iota$. Therefore $H$ is an abelian group. Moreover, $\mathrm{Fix}_{R}(\iota)={\bf I}(R)$ and $\mathrm{Fix}_R(\varphi^\ell)=\cent R\varphi$ for every $\ell\in \{1,\ldots,p-1\}$.

\smallskip

\noindent\textsc{Case 1: }$R$ is abelian of exponent at most $2$.

\smallskip

\noindent As $R$ has exponent at most $2$, $\iota$ is the identity permutation and hence $H=\langle\varphi\rangle$ is cyclic of prime order $p$. From~\eqref{eq:ocl} and from the fact that $|\cent R\varphi|\le |R|/2$, we obtain
$$o=\frac{1}{p}(|R|+(p-1)|\cent R\varphi|)\le 
\frac{1}{p}\left(|R|+(p-1)\frac{|R|}{2}\right)
=\frac{(p+1)|R|}{2p}\le\frac{3|R|}{4}=|R|-\frac{|R|}{4}$$
and the lemma follows in this case because $\mathbf{c}(R)=(|R|+|{\bf I}(R)|)/2=|R|$.

\smallskip

In particular, for the rest of the argument we suppose that $R$ has exponent greater than $2$. Thus $H$ has order $2p$.

\smallskip

\noindent\textsc{Case 2: }$p$ is odd.

\smallskip

\noindent As $H$ is abelian of order $2p$, we deduce that $H$ is cyclic  and $\mathrm{Fix}_R(\iota\varphi^{\ell})=\cent R {\varphi^\ell}\cap \mathrm{Fix}_R(\iota)=\cent R\varphi\cap  {\bf I}(R)$, for every $\ell\in \{1,\ldots,p-1\}$. Now,~\eqref{eq:ocl} yields (in the second inequality we are using the fact that $\varphi$ is not the identity automorphism and hence $|\cent R\varphi|\le |R|/2$)
\begin{eqnarray*}
o&=&\frac{1}{2p}\left(|R|+|{\bf I}(R)|+(p-1)|\cent R \varphi|+(p-1)|\cent R\varphi\cap {\bf I}(R)|\right)\\
&\le&\frac{1}{2p}\left(|R|+|{\bf I}(R)|+(p-1)|\cent R \varphi|+(p-1)|{\bf I}(R)|\right)=\frac{|R|+|{\bf I}(R)|}{2}-\frac{|R|}{2}+\frac{1}{2p}\left(|R|+(p-1)|\cent R \varphi|\right)\\
&\le &\mathbf{c}(R)-\frac{|R|}{2}+\frac{1}{2p}\left(|R|+(p-1)\frac{|R|}{2}\right)=
\mathbf{c}(R)-|R|\left(\frac{1}{2}-\frac{p+1}{4p}\right)\\
&=&\mathbf{c}(R)-|R|\frac{p-1}{4p}\le \mathbf{c}(R)-\frac{|R|}{6}.
\end{eqnarray*}

\smallskip

For the rest of the argument we may suppose $p=2$. If $\varphi=\iota$, then $R$ is an abelian group of exponent greater than $2$ and we obtain that part~\eqref{eq:aut02} holds in this case. Therefore, we may suppose that $\varphi\ne\iota$. As $H=\langle\varphi,\iota\rangle$ is abelian of order $2p$, we deduce that $H=\{id_R,\iota,\varphi,\iota\varphi\}$ is elementary abelian of order $4$. Moreover, $\mathrm{Fix}_R(\iota)={\bf I}(R)$, $\mathrm{Fix}_R(\varphi)=\cent R\varphi$ and $\mathrm{Fix}_R(\iota\varphi):=\cent R\varphi^{-1}$. Thus
$$o=\frac{1}{4}\left(|R|+|{\bf I}(R)|+|\cent R\varphi|+|\cent R\varphi^{-1}|\right).$$
From Lemmas~\ref{l:aut0} and~\ref{l:aut-1}, we may suppose that $|R:\cent R\varphi|\ge 4$.

Miller~\cite{Miller} has shown that a non-identity automorphism of a non-abelian group inverts at most $3|R|/4$ elements. Therefore, $|\cent R\varphi^{-1}|\le 3|R|/4$. Observe that the same inequality holds when $R$ is abelian because $\cent R\varphi^{-1}$ is a proper subgroup of $R$ and hence $|\cent R\varphi^{-1}|\le |R|/2\le 3|R|/4$. In particular, if $|R:\cent R \varphi|\ge 5$, then we deduce
$$o=\frac{1}{4}\left(|R|+|{\bf I}(R)|+\frac{|R|}{5}+\frac{3|R|}{4}\right)=\frac{1}{4}\left(\frac{39}{20}|R|+|{\bf I}(R)|\right)\le \frac{|R|+|{\bf I}(R)|}{2}-\frac{|R|}{80}=\mathbf{c}(R)-\frac{|R|}{80}.$$
For the rest of the argument we may suppose that $|R:\cent R\varphi|\le 4$ and hence $|R:\cent R\varphi|=4$. Therefore
\begin{equation}\label{eq:lead}
o=\frac{1}{4}\left(\frac{5|R|}{4}+|{\bf I}(R)|+|\cent R\varphi^{-1}|\right).
\end{equation}

Assume $|\cent R\varphi^{-1}|\le 2|R|/3$. Then, from~\eqref{eq:lead}, we get 
$$o=\frac{1}{4}\left(\frac{5|R|}{4}+|{\bf I}(R)|+\frac{2|R|}{3}\right)=\frac{1}{4}\left(\frac{23}{12}|R|+|{\bf I}(R)|\right)\le \frac{|R|+|{\bf I}(R)|}{2}-\frac{|R|}{48}=\mathbf{c}(R)-\frac{|R|}{48}.$$
Therefore, we may assume that $|\cent R\varphi^{-1}|>2|R|/3$.

As $2|R|/3<|\cent R\varphi^{-1}|\le 3|R|/4$, from~\cite[Structure Theorem]{LMac}, we deduce that $$|\cent R\varphi^{-1}|=\frac{3|R|}{4}$$ and that $R$ contains an abelian subgroup $A$  with $|R:A|=|A:\cent Ax|=2$, for every $x\in R\setminus A$.

Suppose that $A$ is not $\varphi$-invariant. Since $\varphi$ has order $p=2$, $A\cap A^\varphi$ has index $4$ in $R$ and is $\varphi$-invariant. Observe that $R/(A\cap A^\varphi)$ is an elementary abelian $2$-group of order $4$. Let $T$ be the index $2$ subgroup of $R$ containing $A\cap A^\varphi$ and with $A\ne T\ne A^\varphi$. We have $$\cent R\varphi^{-1}=(\cent R\varphi^{-1}\cap A)\cup (\cent R\varphi^{-1}\cap A^\varphi)\cup (\cent R\varphi^{-1}\cap T).$$
Let $a\in \cent R\varphi^{-1}\cap A$. Then $a^{-1}=a^{\varphi}\in A^\varphi\cap A$ and hence $\cent R\varphi^{-1}\cap A=\cent {A\cap A^\varphi}\varphi^{-1}$ and (similarly) $\cent R\varphi^{-1}\cap A^\varphi=\cent {A\cap A^\varphi}\varphi^{-1}$. Therefore
$$\cent R\varphi^{-1}=\cent {A\cap A^\varphi}\varphi^{-1}\cup (\cent R\varphi^{-1}\cap T).$$
We deduce
$$|\cent R\varphi^{-1}|=|\cent {A\cap A^\varphi}\varphi^{-1}|+|\cent R\varphi^{-1}\cap (T\setminus(A\cap A^\varphi))|\le |A\cap A^\varphi|+(|T|-|A\cap A^\varphi|)=|T|=\frac{|R|}{2};$$
however this contradicts $|\cent R\varphi^{-1}|=3|R|/4$. Thus $A$ is $\varphi$-invariant.

\smallskip

\noindent\textsc{Case: }$\varphi$ 
inverts each element in $A$, that is, $a^\varphi=a^{-1}$, for every $a\in A$.

\smallskip

\noindent As $|\cent R\varphi^{-1}|=3|R|/4>|R|/2=|A|$, there exists $x\in R\setminus A$ with $x^\varphi=x^{-1}$. It follows that $\cent R\varphi^{-1}=A\cup \cent A xx$ and hence 
\begin{equation}\label{eq:art5}
|\cent R\varphi^{-1}|=|A|+\frac{|A|}{2}.
\end{equation} A computation gives $\cent R \varphi={\bf I}(A)\cup \{ax\mid a\in A, a^2=x^{-2}\}$. Let $a,b\in A$ with the property that $a^2=x^{-2}=b^{2}$. Then $(ab^{-1})^2=a^2b^{-2}=x^{-2}x^2=1$. This shows that either  $\{ax\mid a\in A, a^2=x^{-2}\}$ is the empty set or $\{ax\mid a\in A, a^2=x^{-2}\}=\{b\bar{a}x\mid b\in {\bf I}(A)\}$, where $\bar{a}\in A$ is a fixed element with $\bar{a}^2=x^{-2}$. In particular, $|\cent R\varphi|\in\{|{\bf I}(A)|,2|{\bf I}(A)|\}$. As $|R:\cent R\varphi|=4$, we deduce that either $|A:{\bf I}(A)|=2$ and  $\{ax\mid a\in A, a^2=x^{-2}\}=\emptyset$, or $|A:{\bf I}(A)|=4$ and  $\{ax\mid a\in A, a^2=x^{-2}\}\ne\emptyset$. In the first case, from~\eqref{eq:art5}, we have $$|\cent R\varphi^{-1}|=|A|+|A|/2=|A|+|{\bf I}(A)|\le \frac{|R|}{2}+|{\bf I}(R)|.$$
Thus
\begin{eqnarray*}
o&\le& \frac{1}{4}\left(\frac{5}{4}|R|+|{\bf I}(R)|+\frac{|R|}{2}+|{\bf I}(R)|\right)\le 
\frac{1}{4}\left(\frac{7}{4}|R|+2|{\bf I}(R)|\right)\\ 
&=&\frac{|R|+|{\bf I}(R)|}{2}-\frac{|R|}{16}=\mathbf{c}(R)-\frac{|R|}{16}.
\end{eqnarray*}
In the second case, from~\eqref{eq:art5}, we have $$|\cent R\varphi^{-1}|=|A|+|A|/2=|A|+2|{\bf I}(A)|=\frac{|R|}{2}+|{\bf I}(A)|+|{\bf I}(A)|=\frac{|R|}{2}+\frac{|R|}{8}+|{\bf I}(A)|\le \frac{5|R|}{8}+|{\bf I}(R)|.$$
Thus
\begin{eqnarray*}
o&\le& \frac{1}{4}\left(\frac{5}{4}|R|+|{\bf I}(R)|+\frac{5|R|}{8}+|{\bf I}(R)|\right)\le 
\frac{1}{4}\left(\frac{15}{8}|R|+2|{\bf I}(R)|\right)\\ 
&=&\frac{|R|+|{\bf I}(R)|}{2}-\frac{|R|}{32}=\mathbf{c}(R)-\frac{|R|}{32}.
\end{eqnarray*}

\smallskip

It remains to deal with the case that $\varphi$ does not invert each element of $A$. Observe that $\cent A\varphi^{-1}$ is a subgroup of $A$ because $A$ is abelian.
 In particular, $|\cent R\varphi^{-1}\cap A|\le |A|/2=|R|/4$. As $|\cent R\varphi^{-1}|=3|R|/4$, we deduce that  $\varphi$ inverts each element in $R\setminus A$ and $|\cent R\varphi^{-1}\cap A|=|R|/4$.

\smallskip

\noindent\textsc{Case: }$\varphi$ 
inverts each element in $R\setminus A$.

\smallskip

\noindent Fix $x\in R\setminus A$. In particular, for every $a\in A$, we have $x^{-1}a^{-1}=(ax)^{-1}=(ax)^\varphi=a^\varphi x^\varphi=a^\varphi x^{-1}$ and hence $a^\varphi=x^{-1}a^{-1}x$. From this it follows $$\cent R\varphi^{-1}=\cent Ax\cup Ax \textrm{ and }\cent R\varphi=\cent A{\iota_x}^{-1}\cup {\bf I}(R\setminus A),$$ where ${\bf I}(R\setminus A):=\{m\in R\setminus A\mid m^2=1\}$. 

Suppose that ${\bf I}(R\setminus A)=\emptyset$. Then $\cent R\varphi=\cent A{\iota_x}^{-1}$ and hence $|A:\cent A{\iota_x}^{-1}|=2$ because $|R:\cent R\varphi|=4$. As $|A:\cent A x|=2$, we deduce that $|A:\cent A{x}\cap \cent {A}{\iota_x}^{-1}|\le 4$. Clearly, $\cent A x\cap \cent A{\iota_x}^{-1}\subseteq {\bf I}(A)$ and hence $|A:{\bf I}(A)|\le 4$. We deduce
$$|\cent R\varphi|=|\cent A{\iota_x}^{-1}|=|\cent A{\iota_x}^{-1}\cap (A\setminus \cent A x)|+|\cent A{\iota_x}^{-1}\cap \cent Ax|\le \frac{|A|}{4}+|{\bf I}(A)|\le \frac{|R|}{8}+|{\bf I}(R)|.$$
Thus
\begin{eqnarray*}
o&=& \frac{1}{4}\left(|R|+|{\bf I}(R)|+|\cent R\varphi|+|\cent R\varphi^{-1}|\right)=
\frac{1}{4}\left(\frac{7|R|}{4}+|{\bf I}(R)|+|\cent R\varphi|\right)\le
\frac{1}{4}\left(\frac{7|R|}{4}+|{\bf I}(R)|+\frac{|R|}{8}+|{\bf I}(R)|\right)\\
&=&
\frac{1}{4}\left(\frac{15|R|}{8}+2|{\bf I}(R)|\right)\le
\frac{|R|+|{\bf I}(R)|}{2}-\frac{|R|}{32}=\mathbf{c}(R)-\frac{|R|}{32}.
\end{eqnarray*}

Suppose that ${\bf I}(R\setminus A)\ne \emptyset$. In particular, we may suppose that $x\in {\bf I}(R\setminus A)$, that is, $x^2=1$. From this it follows that 
\begin{align*}
\cent R\varphi&=\cent A{\iota_x}^{-1}\cup \cent A{\iota_x}^{-1}x,\\
\cent R\varphi^{-1}&=\cent Ax\cup Ax,\\
{\bf I}(R)&={\bf I}(A)\cup \cent A{\iota_x}^{-1}x.
\end{align*}
As $|\cent R\varphi|=|R|/4$, we deduce $|\cent A{\iota_x}^{-1}|=|A|/4$.
Assume  that $|{\bf I}(R)|\ge |\cent R \varphi|$, that is, $|{\bf I}(A)|\ge |\cent A{\iota_x}^{-1}|$. Thus
\begin{eqnarray*}
o&=& \frac{1}{4}\left(|R|+|{\bf I}(R)|+|\cent R\varphi|+|\cent R\varphi^{-1}|\right)=
\frac{1}{4}\left(\frac{7|R|}{4}+2|{\bf I}(R)|\right)\le
\frac{|R|+|{\bf I}(R)|}{2}-\frac{|R|}{16}=\mathbf{c}(R)-\frac{|R|}{16}.
\end{eqnarray*}
Assume that $|{\bf I}(R)|<|\cent R\varphi|$, that is, $|{\bf I}(A)|<|\cent A{\iota_x}^{-1}|$. Observe now
$$\cent Ax\cap {\bf I}(A)=\cent A{\iota_x}^{-1}\cap {\bf I}(A)=\cent Ax\cap \cent A{\iota_x}^{-1}.$$
As $|{\bf I}(R)|<|\cent R\varphi|$, from these equalities we deduce ${\bf I}(A)=\cent Ax\cap \cent A{\iota_x}^{-1}$ and that $\cent Ax\ne \cent A{\iota_x}^{-1}$. Moreover,
$$|{\bf I}(A)|=\frac{|A|}{8},\,|\cent Ax|=\frac{|A|}{2},\,|\cent A{\iota_x}^{-1}|=\frac{|A|}{4}.$$
In particular, $\mathbf{c}(R)=(|R|+|{\bf I}(R)|)/2=19|R|/32$. Thus
\begin{eqnarray*}
o&=& \frac{1}{4}|R|\left(1+\frac{3}{16}+\frac{1}{4}+\frac{3}{4}\right)=\frac{35|R|}{64}=\frac{19|R|}{32}-\frac{3|R|}{64}=\mathbf{c}(R)-\frac{3|R|}{64}.
\end{eqnarray*}
\end{proof}

\begin{proposition}\label{propo:aut}Let $R$ be a finite group and suppose that $R$ is not an abelian group of exponent greater than $2$ and that $R$ is not a generalized dicyclic group. Then the set
$$\{S\subseteq R\mid S=S^{-1}, R<\norm{\Aut(\Cay(R,S))}{R}\}$$
has cardinality at most $2^{\mathbf{c}(R)-|R|/96+(\log_2|R|)^2}$.
\end{proposition}
\begin{proof}
Since a chain of subgroups of $R$ has length at most $\log_2 (|R|)$, $R$ has a
generating set of cardinality at most $\lfloor \log_2 (|R|)\rfloor \le \log_2 (|R|)$. Any automorphism of $R$ is uniquely determined by its action on the elements of a generating set for $R$. Therefore $| \Aut(R)| \le |R|^{\lfloor \log_2 (|R|)\rfloor} \le 2^{(\log_2 (|R|))^2}$.
Now the proof follows from Lemma~\ref{l:aut} and from the fact that we have at most $|\Aut(R)|\le 2^{(\log |R|)^2}$ choices for the non-identity automorphism $\varphi$ of $R$.
\end{proof}
\section{Proofs of Theorems~$\ref{thrm:1}$ and~$\ref{equivalence}$}

\begin{proof}[Proof of Theorem~$\ref{thrm:1}$]
Let $R$ be a finite group of order $r$ which is neither generalized dicyclic nor abelian of exponent greater than $2$. 
By Lemma~\ref{lemma111} and Proposition~\ref{propo:aut}, we have
$$\lim_{r\to\infty}\frac{|\{S\subseteq R\mid S=S^{-1}, R<\norm{\Aut(\Cay(R,S))}{R}\}|}{|\{S\subseteq R\mid S=S^{-1}\}|}\le\lim_{r\to\infty}2^{-\frac{r}{96}+(\log_2(r))^2}=0.\qedhere$$
\end{proof}

\begin{proof}[Proof of Theorem~$\ref{equivalence}$]
Let $R$ be a finite group. It was shown in~\cite{DSV,MSV} that Xu conjecture holds true when $R$ is a generalized dicyclic group or when $R$ is an abelian group of exponent greater than $2$. In particular, for the rest of the proof we may assume that $R$ is neither a generalized dicyclic group nor an abelian group of exponent greater than $2$.

 Let us denote by $\mathcal{N}(R):=\{S\subseteq R\mid S=S^{-1}, R\unlhd \Aut(\Cay(R,S))\}$, $\mathcal{C}(R):=\{S\subseteq R\mid S=S^{-1}, R=\Aut(\Cay(R,S))\}$ and $\mathcal{T}(R):=\{S\subseteq R\mid S=S^{-1}\}$.  If Conjecture~\ref{conjecture....} holds true, then 

$$\lim_{|R|\to\infty}\frac{|\mathcal{C}(R)|}{|\mathcal{T}(R)|}=1$$
and hence 
$$\lim_{|R|\to\infty}\frac{|\mathcal{N}(R)|}{|\mathcal{T}(R)|}=1,$$
because $\mathcal{C}(R)\subseteq \mathcal{N}(R)$, that is, Conjecture~\ref{conjectureXu} holds true. Conversely, suppose that Conjecture~\ref{conjectureXu} holds true, that is, $\lim_{|R|\to\infty}|\mathcal{N}(R)|/|\mathcal{T}(R)|=1$. Now,
\begin{eqnarray*}
\mathcal{N}(S)&=&\mathcal{C}(S)\cup\{S\subseteq R\mid S=S^{-1}, R\unlhd \Aut(\Cay(R,S)),R<\Aut(\Cay(R,S))\}\\
&\subseteq &\mathcal{C}(S)\cup\{S\subseteq R\mid S=S^{-1}, R<\norm{\Aut(\Cay(R,S))}R\}
\end{eqnarray*}
and hence, by Theorem~\ref{thrm:1}, we have
$$1=\lim_{|R|\to\infty}\frac{|\mathcal{N}(R)|}{|\mathcal{T}(R)|}\le
\lim_{|R|\to\infty}\frac{|\mathcal{C}(R)|}{|\mathcal{T}(R)|}+\lim_{|R|\to\infty}\frac{|\{S\subseteq R\mid S=S^{-1}, R<\norm{\Aut(\Cay(R,S))}R\}|}{|\mathcal{T}(R)|}=
\lim_{|R|\to\infty}\frac{|\mathcal{C}(R)|}{|\mathcal{T}(R)|},
$$
that is, Theorem~\ref{conjecture....} holds true. 

\end{proof}

\end{document}